\documentclass[10pt]{amsart}
\usepackage{amsmath,amssymb,amsthm}
\usepackage{graphicx,url,color}
\newtheorem{theorem}{Theorem}    

\newtheorem{proposition}{Proposition} 
 
\newtheorem{question}{Question} 
\newtheorem{corollary}{Corollary}
 
\theoremstyle{definition}
\newtheorem{definition}{Definition}
\newtheorem{example}{Example}

\newtheorem*{remark*}{Remark}
\newcommand{\Z}{\mathbb{Z}}
\newcommand{\R}{\mathbb{R}}

\newcommand{\Tor}{\mathrm{Tor}}

\newcommand{\scl}{\mathrm{scl}}
\newcommand{\Ker}{\mathrm{Ker}\,}

 \makeatletter
    
    \@addtoreset{equation}{section}
  \makeatother

\title[Group whose generalized torsions are torsion]{On a group whose generalized torsion elements are torsion elements}
\author[T.Ito]{Tetsuya Ito}
\address{Department of Mathematics, Kyoto University, Kyoto 606-8502, JAPAN}
\email{tetitoh@math.kyoto-u.ac.jp}
\subjclass[2020]{}
\keywords{}

\begin{document}

\begin{abstract}
We show that a group whose generalized torsion elements are torsion elements (which we call a $TR^{*}$-group) is a torsion-by-$R^{*}$ group, an extension of a torsion group by a group without generalized torsion elements. We also discuss a generalized torsion group, a group all of whose non-trivial elements are generalized torsion elements.
\end{abstract}

\maketitle

\section{Introduction}

A non-trivial element $g$ of a group $G$ is a \emph{generalized torsion element} if there exists a positive integer $n$ and $x_1,\ldots,x_n \in G$ such that they satisfy
\begin{equation*}
\label{eqn:gt}g^{x_1} g^{x_2}\cdots g^{x_n} = 1.
\end{equation*} 
Here $g^{x}:=xgx^{-1}$.

A group $G$ is an \emph{$R^{*}$-group} (or, $\Gamma$-torsion-free group) if it has no generalized torsion elements. Such a group has been studied due to a close connection to (bi-)orderable groups (see \cite[Section 4]{MR}, for examples). 

Unlike usual torsion elements, a generalized torsion element is much more complicated. It is often hard to check whether a given element is a generalized torsion element or not. For example, only recently it is proven that the free product is an $R^{*}$-group if and only if all the factors are $R^{*}$-groups \cite{IMT}.

When one studies $R^{*}$-groups, it is usually assumed that a group is torsion-free, because a torsion element is obviously a generalized torsion element.
However, when we study generalized torsion elements (such as, the classification of all generalized torsion elements of a group $G$) it is more natural to allow a group with torsion.

As a natural generalization of $R^{*}$-groups, we define the following class of groups. We say that a generalized torsion element $g$ is \emph{genuine} if it is not a torsion element.

\begin{definition}[$TR^{*}$-group]
A group $G$ is an \emph{$TR^{*}$-group} if $G$ has no genuine generalized torsion element. That is, every generalized torsion element of $G$ is a torsion element.
\end{definition}

Since the condition of torsion or generalized torsion are local, a group $G$ is an $R^{*}$-group (resp. $TR^{*}$-group) if and only if it is locally so (i.e. every finitely generated subgroup of $G$ is an $R^{*}$-group (resp. $TR^{*}$-group)). 

An $R^{*}$-group is understood as a torsion-free $TR^{*}$-group. As an opposite extreme, a \emph{torsion group}, a group all of whose elements are torsion elements, is obviously a $TR^{*}$-group. We show that, $TR^{*}$ groups are extensions of these two extreme cases.

\begin{theorem}
\label{theorem:RT}
A group $G$ is a $TR^{*}$-group if and only if it is a torsion-by-$R^{*}$ group; there exists a normal subgroup $K$ which is a torsion group such that $G\slash K$ is an $R^{*}$-group. 
\end{theorem}

In \cite{BSS} it is shown that FC-groups (a group in which every element has finitely many conjugates) and nilpotent groups are $TR^{*}$-groups.
These can be deduced from Theorem \ref{theorem:RT}; for such groups $\Tor(G)$, the set of torsion elements of $G$, is a normal subgroup and $G\slash \Tor(G)$ is an $R^{*}$-group so $G$ is a $TR^{*}$-group.

More generally, we show that an extension of torsion groups has the following nice property concerning genuine generalized torsion elements.

\begin{theorem}
\label{theorem:T-by-G}
Let $G$ be an extension\footnote{In this paper we adopt the convention that \emph{$G$ is an extension of $K$ by $Q$} if there is an exact sequence $1 \rightarrow K \rightarrow G \rightarrow Q \rightarrow 1$. } of a torsion group $K$ by a group $Q$; there is a surjection $p:G\rightarrow Q$ whose kernel is $K$.
Then $g \in G$ is a genuine generalized torsion element if and only if $p(g) \in Q$ is a genuine generalized torsion element.
\end{theorem}

We also prove the following mild generalization of Theorem \ref{theorem:T-by-G}.

\begin{theorem}
\label{theorem:GT-group}
Let $G$ be an extension of a group $K$ by a group $Q$, so there is a surjection $p:G\rightarrow Q$ whose kernel is $K$. 
If all the non-trivial element of $K$ is a generalized torsion elements of $G$, then  $g \in G$ is a generalized torsion element if and only if $p(g) \in Q$ is a genuine generalized torsion element of $Q$ or $p(g)=1$.
\end{theorem}

Thus in a study of genuine generalized torsion elements it is sufficient to assume that the group has no torsion normal subgroups. Similarly, when we study a generalized torsion element of a group $G$, we can assume that $G$ has no normal subgroup that consists of generalized torsion elements.

This motivates us to study the following natural generalization of torsion groups.
\footnote{We remark, however, that the normal subgroup $K$ in Theorem \ref{theorem:GT-group} is not necessarily a generalized torsion group, because we require an element $k \in K \subset G$ to be a generalized torsion element of $G$, but we do not require $k$ to be a generalized torsion element of $K$.}.

\begin{definition}[Generalized torsion group]
A group $G$ is a \emph{generalized torsion group} if every non-trivial element of $G$ is a generalized torsion element.
\end{definition}

We observe properties of generalized torsion group in Section \ref{sec:gt-group}. Among them, we show that a finitely generated generalized torsion group is freely indecomposable except one particular example.

\begin{theorem}
\label{theorem:gt}
A finitely generated generalized torsion group $G$ is decomposable with respect to free product if and only if $G = D_{\infty}=\Z_2 \ast \Z_2$ is an infinite dihedral group. 
\end{theorem}

\section*{Acknowledgement}
The author is partially supported by JSPS KAKENHI Grant Numbers 19K03490, 21H04428.
He would like to thank R. Bastos, S. Schneider, and D. Silveira for their stimulating works and discussions, and to R. Coulon and Y. Antol\'in for sharing their insight and answers to the Questions in the first version.

\section{Characterization of $TR^{*}$-group}

Our characterization of $TR^{*}$-group is based on the following rather simple observation. 

\begin{proposition}
\label{prop:prod-torsion}
Let $a$ and $b$ be torsion elements of a group $G$. Then $ab$ is a generalized torsion element.
\end{proposition}
\begin{proof}
Let $a^{m}=b^{n}=1$ for $m,n>0$. Then
\[ (ab)(b^{n-1}(ab)b^{1-n})(b^{n-2}(ab)b^{2-n}) \cdots (b^{n-(m-1)}(ab) b^{(n-1)-m}) = a^{m}b^{(n-m)} = b^{n-m}\]
so $ab$ is a generalized torsion element.
\end{proof}

Thus when $G$ has no torsion normal subgroups, a product of torsion elements provides a lot of genuine generalized torsion elements. Thus a structure of genuine generalized torsion elements is more complicated.

Beside a construction of generalized torsion elements from torsion elements, Proposition \ref{prop:prod-torsion} implies the following important property.

\begin{corollary}
\label{cor:tor}
If $G$ is a $TR^{*}$-group then $\Tor(G)$, the set of all torsion elements of $G$, is a normal subgroup of $G$.
\end{corollary}

Using this property of $TR^{*}$-groups, we prove our main theorems.

\begin{proof}[Proof of Theorem \ref{theorem:RT}]
First we show that if a group $G$ is a torsion-by-$R^{*}$ group, then $G$ is a $TR^{*}$-group. Let $1 \rightarrow K \rightarrow G \stackrel{p}{\rightarrow} Q \rightarrow 1$ where $K$ is a torsion group and $Q$ is an $R^{*}$-group.
Let $g \in G$ be a generalized torsion element. If $p(g)\neq 1$, then $p(g) \in Q$ is a generalized torsion element of $Q$. This is impossible since $Q$ is an $R^{*}$-group. Thus $p(g)=1$, namely, $g \in \Ker p =K$. Since $K$ is a torsion group, $g$ is a torsion element.

Conversely, let $G$ be a $TR^{*}$-group.
By Corollary \ref{cor:tor}, $\Tor(G)$ is a normal subgroup of $G$. We show that the quotient group $G \slash \Tor(G)$ is an $R^{*}$-group.
Assume to the contrary that there exists a generalized torsion element $q \in G\slash \Tor(G)$.
Thus there exists $y_1,\ldots,y_k \in G\slash \Tor(G)$ such that 
$q^{y_1}q^{y_2}\cdots q^{y_k}=1$. Take $g, x_1,\ldots,x_n \in G$ so that $p(g)=q$, $p(x_i)=y_i$. Then $p(g^{x_1}g^{x_2}\cdots g^{x_n})=q^{y_1}q^{y_2}\cdots q^{y_n}=1$ so $g^{x_1}g^{x_2}\cdots g^{x_n} \in \Tor(G)$. 
Thus $(g^{x_1}g^{x_2}\cdots g^{x_{n}})^{m}=1$ for some $m \geq 1$, which means that $g$ is a generalized torsion element of $G$. Since $G$ is a $TR^{*}$-group, $g$ is a torsion element. However, this means that $g \in \Tor(G)$, which is a contradiction.
\end{proof}

\begin{proof}[Proof of Theorem \ref{theorem:T-by-G}]
Let $g \in G$ be a genuine generalized torsion element.
Then $p(g) \neq 1$ because otherwise $g \in \Ker p = K$ is a torsion element.
Thus $p(g)$ is a generalized torsion element of $Q=G\slash K$. 
$p(g)$ cannot be a torsion element because otherwise $p(g)^{m}=p(g^{m})=1$ for some $m>1$ so $g^{m} \in \Ker p = K$ is a torsion element (hence $g$ is a torsion).

Conversely, assume that $p(g)$ is a genuine generalized torsion element. Then $g$ is not a torsion element, and $p(g)^{y_1}p(g)^{y_2}\cdots p(g)^{y_n}=1$ for some $y_1,\ldots,y_n \in Q$. Take $ x_1,\ldots,x_n\in G$ so that $p(x_i)=y_i$. Then $p(g^{x_1}g^{x_2}\cdots g^{x_n})=1$ so $g^{x_1}g^{x_2}\cdots g^{x_n} \in K$. Thus  $(g^{x_1}g^{x_2}\cdots g^{x_n})^{m}=1$ for some $m \geq 1$ so $g$ is a genuine generalized torsion element.
\end{proof}

We remark that Theorem \ref{theorem:T-by-G} contains the following generalization of Proposition \ref{prop:prod-torsion} (which is not hard to see directly); If $G$ is an extension of a torsion group $K$ by a group $Q$ and $g \in G$ is a genuine generalized torsion element, for every $k \in K$, $kg$ is again a generalized torsion element.

The proof of Theorem \ref{theorem:GT-group} is quite similar to that of Theorem \ref{theorem:T-by-G}.

\begin{proof}[Proof of Theorem \ref{theorem:GT-group}]
`Only if' direction is obvious so we prove `if' direction. 
If $p(g)=1$ then $g \in K$ so $g$ is a generalized torsion element.
If $p(g) \neq 1 \in Q$ is a generalized torsion element of $Q$, then $p(g)^{y_1}p(g)^{y_2}\cdots p(g)^{y_n}=1$ for some $y_1,\ldots,y_n \in Q$. Take $ x_1,\ldots,x_n \in G$ so that $p(x_i)=y_i$. Then $p(g^{x_1}g^{x_2}\cdots g^{x_n})=1$ so $g^{x_1}g^{x_2}\cdots g^{x_n} \in K$. Let $k = (g^{x_1}g^{x_2}\cdots g^{x_n}) \in K$. By assumption, $k$ is a generalized torsion element of $G$ so $k^{z_1}k^{z_2}\cdots k^{z_{m}}=1$ for some $z_1,\ldots,z_m \in G$. Therefore
\begin{align*}
1& = k^{z_1}k^{z_2}\cdots k^{z_{m}} \\
&= (g^{x_1}g^{x_2}\cdots g^{x_n})^{z_1}(g^{x_1}g^{x_2}\cdots g^{x_n})^{z_2} \cdots (g^{x_1}g^{x_2}\cdots g^{x_n})^{z_m} \\
&= (g^{x_1z_1}g^{x_2z_1}\cdots g^{x_nz_1})(g^{x_1z_2}g^{x_2z_2}\cdots g^{x_nz_2})\cdots (g^{x_1z_m}g^{x_2z_m}\cdots g^{x_nz_m})
\end{align*}
so $g$ is a generalized torsion element.
\end{proof}

As in Theorem \ref{theorem:T-by-G}, Theorem \ref{theorem:GT-group} says that if both $g \in G$ and $k \in K$ are generalized torsion elements of $G$, then their product $gk$ is again a generalized torsion element whenever $K$ is a normal subgroup consisting of generalized torsion elements \emph{of $G$}. 

Here we give one example where Theorem \ref{theorem:GT-group} is applied.

\begin{example}
\label{example:gt-normal}
Let $G$ be a torsion-free group and $K$ be an infinite cyclic normal subgroup generated by $k$. Assume that $K$ is not central. Thus there exists $g \in G$ such that $k^{g}=gkg^{-1}=k^{m}$ for some $m \neq 0,1$. If $m<0$, then $K$ satisfies the assumption of Theorem \ref{theorem:GT-group}; for every $n>0$, $k^{n}$ is a generalized torsion element because $(gkg^{-1})^{n} k^{|m|n} = 1$. 
\end{example}

Unfortunately, the group for which Theorem \ref{theorem:GT-group} can be effectively applied is limited. This is not surprising because it is rare that the products of generalized torsion elements, such as, powers of a generalized torsion element, become a generalized torsion element.

In an opposite direction, we show that a torsion-free hyperbolic group, one of the most famous and important classes of groups, never admits a normal subgroup consisting of generalized torsion elements.

The argument below uses $\scl_G$, the \emph{stable commutator length}. 
Since we only use various known results on stable commutator length, we do not explain details. The reader can simply view the stable commutator length as a map\footnote{Strictly speaking, $\scl_G$ is usually defined on $[G,G]$. We extend the domain of definition by defining $\scl_G(g)=\frac{\scl(g^k)}{k}$ if $g^{k} \in [G,G]$ for some $k>0$, and $\scl_G(g)=\infty$ otherwise.} $\scl_G:G \rightarrow \R_{\geq 0} \cup \{\infty\}$ that is conjugation-invariant (i.e. $\scl_G(ghg^{-1})=\scl_G(h)$ for all $g,h \in G$) and homogeneous (i.e. $\scl_G(g^{n}) = |n| \scl(g)$ for all $n \in \Z$).
See \cite{Ca} for basics of stable commutator length.

\begin{proposition}
Let $G$ be a word hyperbolic group. Then there exists a constant $N>0$ (that only depends on $G$) such that for every $1 \neq g \in G$ and $m\geq N$, $g^{m}$ is not a generalized torsion element, unless $g^{n}$ is conjugate to $g^{-n}$ for some $n>0$ (in which case, $g^{mn}$ is a generalized torsion element for all $m>0$). 
In particular, if $G$ is a torsion-free word hyperbolic group, then $G$ has no normal subgroup consisting of generalized torsion elements.
\end{proposition}
\begin{proof}
Let $g \in G$ be a generalized torsion element. 
Then its stable commutator length $\scl_G(g)$ satisfies $\scl_G(g)<\frac{1}{2}$ \cite[Theorem 2.4]{IMT}. On the other hand, if $G$ is a torsion-free hyperbolic group, there exists a constant $C>0$ (that only depends on the group $G$) such that $\scl_G(g)\geq C$, unless $g^{n}$ is conjugate to $g^{-n}$ for some $n>0$ \cite[Theorem A]{CF}. Since the stable commutator length $\scl_G(g)$ has the property that $\scl_G(g^{m})=m\, \scl_G(g)$ for every $g \in G$ and $k \in \Z$. These facts means that when $m > \frac{1}{2C}$ then $g^{m}$ cannot be a generalized torsion element.

The non-existence of normal subgroup consisting of generalized torsion elements follows from the fact that for a torsion-free word hyperbolic group $G$, $g^{n}$ is never conjugate to $g^{-n}$ (see \cite[Remark 3.2]{CF}).
\end{proof}

We remark that the argument presented here can be applied for many other classes of groups, as long as a \emph{spectral gap}, an appropriate lower bound of stable commutator length, is known (for example, 3-manifold groups \cite{ChH1})

\section{Generalized torsion group}
\label{sec:gt-group}

It is not hard to find a generalized torsion group which is not a torsion group.
In fact, Example \ref{example:gt-normal} says that suitable (infinite cyclic)-by-(generalized torsion) group is a generalized torsion group.  

\begin{example}
\label{exam:D}
The infinite dihedral group 
\[ D_{\infty} = \Z_2 \ast \Z_2=\{g,k \: | \: g^{2}=1, gkg^{-1}=k^{-1}\}\]
(which is an extension $1 \rightarrow \Z \rightarrow D_{\infty} \rightarrow \Z_2 \rightarrow 1$) is a generalized torsion group which is not a torsion group.
\end{example}

The arguments in the previous sections tell us the following property of generalized torsion groups.
\begin{theorem}
\label{theorem:closure}
The class of generalized torsion group $GT$ has the following properties.
\begin{itemize}
\item[(i)] $GT$ is closed under quotients.
\item[(ii)] $GT$ is closed under direct limits.
\item[(iii)] $GT$ is closed under extensions.
\item[(iv)] $GT$ is \emph{not} closed under subgroups.
\end{itemize}
\end{theorem}
\begin{proof}
(i,ii) are immediate from the definitions. (iii) follows from Theorem \ref{theorem:GT-group}. (iv) is immediate consequence of Example \ref{exam:D}.
\end{proof}

%

Aforementioned result of stable commutator length of generalized torsion elements \cite[Theorem 2.4]{IMT} implies the following useful property of generalized torsion group. 

\begin{proposition}
\label{prop:scl=0}
If $G$ is a generalized torsion group, then $scl_G(g)=0$ for all element $g \in G$.
\end{proposition}

A typical classes of groups with trivial stable commutator length is \emph{amenable} groups. Due to a similarity of the definition of elementary amenable groups and properties in Theorem \ref{theorem:closure}, one may expect generalized torsion group is (elementary) amenable. However, there is a finitely generated torsion group which is not amenable (such as, Tarski monster \cite{Ol}) so a (generalized) torsion group is not necessarily amenable.

As we have seen in Example \ref{exam:D}, the infinite dihedral group $D_{\infty}=\Z_2 \ast \Z_2$ is generalized torsion group. Theorem \ref{theorem:gt} says that $D_{\infty}$ is the unique example of finitely generated decomposable generalized torsion group.

\begin{proof}[Proof of Theorem \ref{theorem:gt}]
By \cite[Theorem 3.1]{Ch}, if $G=A\ast B$ and $g = a_1b_1\cdots a_Lb_L$ ($1\neq a_i \in A$, $1\neq b_i \in B$), then $scl(g) \geq \frac{1}{2}-\frac{1}{N}$, where
\[ N= \min \{ord(a_1),\ldots,ord(a_L), ord(b_1),\ldots, ord(b_L)\} \]
(Here $ord(x)$ represents the order of $x$, and if $x$ is non-torsion $ord(x)=\infty$).
Therefore by Proposition \ref{prop:scl=0}, if $G$ is a generalized torsion group,  both $A$ and $B$ are torsion group of exponent two (i.e. every non-trivial element has order two). 
Since we assumed that $G=A \ast B$ is finitely generated, so are $A$ and $B$. 
Every torsion group $X$ of exponent two is abelian because for $x,y \in X$, $xyx^{-1}y^{-1}=xyxy=(xy)^{2}=1$. Therefore $A\cong \Z_2^{a}$ and $B \cong \Z_2^{b}$ for some $a,b\geq 1$.

However, when $a+b>1$, the free product $G= \Z_2^a \ast \Z_2^b$ (which is a right-angled Coxeter group) has an element with non-trivial stable commutator length by \cite[Corollary 6.18]{ChH2}. Thus $G$ cannot be a generalized torsion group by Proposition \ref{prop:scl=0}.
\end{proof}

In the first version of the paper, the author asked whether there exists a torsion-free generalized torsion groups, and whether every generalized torsion group has the property NF (having no non-abelian free groups).

R. Coulon and A. Yago taught the author that theorem of Osin \cite{Os} answers both questions in the negative. Osin showed that every countable torsion-free group $H$ (so, the non-abelian free group) can be embedded into a torsion-free, 2-generated group $G$ having exactly two conjugacy classes \cite[Corollary 1.3]{Os}. Such a group $G$ is a generalized torsion group because every non-trivial element $g \in G$ is conjugate to $g^{-1}$. 

On the other hand, when we assume $G$ is finitely presented, the problems remain to be interesting.
\begin{question}
\label{ques:torsion-free-fp-gt}
{$ $}
\begin{itemize}
\item[(i)] Is there a torsion-free, finitely presented generalized torsion group ?
\item[(ii)] Does every finitely presented generalized torsion group have the property NF?
\end{itemize}
\end{question}

\end{document}